\PassOptionsToPackage{hidelinks}{hyperref}
\PassOptionsToPackage{nameinlink,noabbrev}{cleveref}
\documentclass{siamart251216}

\usepackage[T1]{fontenc}
\usepackage[utf8]{inputenc}
\usepackage{lmodern}
\usepackage{amsmath,amssymb,amsfonts,mathtools}
\usepackage{mathrsfs}
\usepackage{bm}
\usepackage{tikz}
\usepackage{enumitem}
\usepackage{booktabs}
\usepackage{array}
\usetikzlibrary{arrows.meta,positioning,calc,fit,decorations.pathmorphing,decorations.pathreplacing}
\allowdisplaybreaks
\numberwithin{theorem}{section}
\newsiamthm{conjecture}{Conjecture}
\newsiamremark{remark}{Remark}
\newsiamthm{example}{Example}
\AddToHook{env/conjecture/begin}{\crefalias{theorem}{conjecture}}
\AddToHook{env/example/begin}{\crefalias{theorem}{example}}
\AddToHook{env/remark/begin}{\crefalias{theorem}{remark}}

\crefname{conjecture}{conjecture}{conjectures}
\Crefname{conjecture}{Conjecture}{Conjectures}
\crefname{theorem}{theorem}{theorems}
\Crefname{theorem}{Theorem}{Theorems}
\crefname{lemma}{lemma}{lemmas}
\Crefname{lemma}{Lemma}{Lemmas}
\crefname{proposition}{proposition}{propositions}
\Crefname{proposition}{Proposition}{Propositions}
\crefname{corollary}{corollary}{corollaries}
\Crefname{corollary}{Corollary}{Corollaries}
\crefname{definition}{definition}{definitions}
\Crefname{definition}{Definition}{Definitions}
\crefname{remark}{remark}{remarks}
\Crefname{remark}{Remark}{Remarks}
\crefname{example}{example}{examples}
\Crefname{example}{Example}{Examples}
\crefname{figure}{Figure}{Figures}
\Crefname{figure}{Figure}{Figures}

\newcommand{\PP}{\mathcal P}
\newcommand{\QQ}{\mathcal Q}

\newcommand{\TT}{\mathcal T}

\newcommand{\ZZ}{\mathbb Z}

\newcommand{\FFq}{\mathbb F}
\newcommand{\wt}{\operatorname{wt}}
\newcommand{\supp}{\operatorname{supp}}

\newcommand{\cross}{\operatorname{cr}}
\newcommand{\sig}{\operatorname{sig}}

\newcommand{\diag}{\operatorname{diag}}

\newcommand{\defeq}{\mathrel{:=}}

\headers{Excess Obstructions}{Yutong Zhang and Yaoran Yang}

\title{Excess Obstructions and Layer-Contained Certificates for the Hypergraph Nash--Williams--Tutte Conjecture}

\author{Yutong Zhang\thanks{Corresponding author. School of Mathematics, Sichuan University, 24 First Loop Road South Section I, Chengdu, Sichuan 610064, China. Email: \email{yutongzhang@stu.scu.edu.cn}.}
\and Yaoran Yang\thanks{School of Mathematics, Sichuan University, 24 First Loop Road South Section I, Chengdu, Sichuan 610064, China. Email: \email{yangyaoran@stu.scu.edu.cn}.}}

\begin{document}
\maketitle

\begin{abstract}
Guo, Li, Shangguan, Tamo, and Wootters proposed a hypergraph analogue of the Nash--Williams--Tutte theorem, asserting that every $k$-weakly-partition-connected hypergraph admits a $k$-distinguishable tree assignment. We identify a sharp edge-count obstruction to the literal statement. A full tree assignment has $\rho(H)=\sum_{e\in E(H)}(|e|-1)$ labelled graph edges, whereas an ordered decomposition into $k$ spanning trees has exactly $k(t-1)$ edges. Since weak partition connectivity implies only $\rho(H)\ge k(t-1)$, every strict inequality rules out a full decomposition. In particular, for all $t\ge2$, $k\ge1$, and $q\ge1$, the hypergraph consisting of $k+q$ labelled copies of the full hyperedge is $k$-weakly-partition-connected but admits no $k$-distinguishable full tree assignment. We therefore isolate the critical regime $\rho(H)=k(t-1)$ and prove that every tree assignment of a critical $k$-weakly-partition-connected hypergraph admits a full ordered decomposition into $k$ spanning trees. Consequently, the critical conjecture reduces to the existence of some tree assignment having a decomposition whose signature fiber is a singleton. We establish this property for layer-contained certificates, without a star hypothesis or a rank restriction in interior layers. These certificates also imply weak partition connectivity by a quotient-rank argument and are stable under one-vertex sums. Finally, we derive a decomposition-indexed generalized-Laplace expansion for the relevant intersection-matrix minor, together with an exact formula for its collected monomial coefficients; we also give the required row-and-column perfect-shuffle normalization, identify criticality with the square full-assignment row count, and separate the critical conjecture from the additional overfull pruning problem.
\end{abstract}

\begin{center}
\begin{minipage}{0.94\textwidth}
\end{minipage}
\end{center}

\section{Introduction, notation, and main results}\label{sec:intro}

The starting point of this note is the conjecture introduced by Guo, Li, Shangguan, Tamo, and Wootters in their work on Reed--Solomon list decoding~\cite{GLSTW24}. The conjecture was proposed as a hypergraph counterpart of the Nash--Williams--Tutte tree-packing theorem~\cite{NW61,Tutte61}; it is closely related to hypergraph decomposition results such as~\cite{FKK03} and was intended to provide a route to the intersection-matrix conjecture originating in~\cite{ST20}. We retain the labelled-multihypergraph formulation required for that application and keep the relevant edge-count parameter explicit.

Let $H=(V,E)$ be a finite labelled multihypergraph. Hyperedges are occurrences rather than merely subsets: occurrences with the same support remain distinct when their labels are distinct. We assume that every hyperedge occurrence is nonempty. A singleton occurrence has a unique assigned tree, namely the edgeless graph on its sole vertex. It contributes zero to $\rho(H)$ and to $w_H(\PP)$ for every partition $\PP$, and its signature coordinate is identically zero. Thus deleting all singleton occurrences preserves $k$-WPC, criticality, the existence of full tree decompositions, and distinguishability after the corresponding zero coordinates are omitted. Conversely, the assigned graph and every ordered decomposition for the reduced hypergraph extend canonically when the edgeless singleton label classes are reinstated; in a layer-contained certificate, the reinstated singleton occurrences may be assigned arbitrary layer indices. We henceforth assume that every hyperedge occurrence has cardinality at least two. Write $t=|V|$. The case $t=1$ is immediate under the convention that the edgeless graph on one vertex is a tree, so throughout the substantive arguments we assume $t\ge2$.

For a partition $\PP$ of $V$ and a hyperedge occurrence $e\in E$, let $p_{\PP}(e)$ denote the number of blocks of $\PP$ met by $e$. Define the weak partition excess by $w_H(\PP)=\sum_{e\in E}(p_{\PP}(e)-1)$. The hypergraph $H$ is \emph{$k$-weakly-partition-connected}, abbreviated $k$-WPC, if
\begin{equation}\label{eq:kwpc}
 w_H(\PP)\ge k(|\PP|-1)\qquad\text{for every partition $\PP$ of $V$}.
\end{equation}
For the discrete partition $\PP_0=\{\{v\}:v\in V\}$, this quantity is the total excess $\rho(H)\defeq\sum_{e\in E}(|e|-1)$. Consequently, $k$-WPC implies $\rho(H)\ge k(t-1)$.

A \emph{tree assignment} of $H$ chooses, for every hyperedge occurrence $e$, a simple tree $F_e$ on vertex set $e$ and labels every edge of $F_e$ by $e$. The resulting graph $G=\biguplus_{e\in E}F_e$ is a labelled multigraph: the union is disjoint at the level of labelled edge occurrences, and parallel edges may arise only from distinct label classes. Necessarily $|E(G)|=\rho(H)$. An ordered \emph{$k$-tree decomposition} of a labelled graph $G$ on $V$ is an ordered partition $E(G)=E(T_0)\dot\cup\cdots\dot\cup E(T_{k-1})$ in which every $T_i$ is a spanning tree of $V$. Throughout, a spanning tree of a multigraph is understood in the graphic-matroid sense; in particular, it contains no pair of parallel edge occurrences. Such a decomposition uses exactly $k(t-1)$ graph edges. We use the adjective \emph{full} when emphasizing that no edge occurrence of the assigned graph is discarded.

For a tree assignment $G$ and an ordered $k$-tree decomposition $\TT=(T_0,\ldots,T_{k-1})$, define its signature $\sig_{\TT}\in\ZZ_{\ge0}^{E(H)}$ by $\sig_{\TT}(e)\defeq\sum_{i=0}^{k-1}i\,|E(T_i)\cap E(F_e)|$ for every hyperedge occurrence $e$. The graph $G$ is \emph{$k$-distinguishable} if it has an ordered $k$-tree decomposition $\TT$ such that every ordered $k$-tree decomposition $\TT'$ with $\sig_{\TT'}=\sig_{\TT}$ satisfies $\TT'=\TT$.

\begin{conjecture}[Hypergraph Nash--Williams--Tutte conjecture, literal form]\label{conj:literal}
For positive integers $t$ and $k$, every $k$-WPC hypergraph on $t$ vertices has a $k$-distinguishable tree assignment.
\end{conjecture}

The first result identifies the exact numerical defect in the literal formulation.

\begin{theorem}[Sharp excess obstruction]\label{thm:excess-obstruction}
Let $H$ be a labelled multihypergraph on $t\ge2$ vertices. If $H$ has a $k$-distinguishable tree assignment, then $\rho(H)=k(t-1)$. Consequently, every $k$-WPC hypergraph satisfying $\rho(H)>k(t-1)$ is a counterexample to \Cref{conj:literal}.
\end{theorem}

\begin{theorem}[Infinite overfull counterexample family]\label{thm:full-edge-counterexamples}
Fix $t\ge2$, $k\ge1$, and $q\ge1$. Let $H_{t,k,q}$ have vertex set $V$ and $k+q$ labelled hyperedge occurrences, each equal to $V$. Then $H_{t,k,q}$ is $k$-WPC and has no $k$-distinguishable tree assignment. The componentwise-minimal parameter triple is $(t,k,q)=(2,1,1)$; it yields the two-vertex hypergraph with two parallel labelled $2$-edges.
\end{theorem}

Consequently, any valid universal theorem about full tree assignments in the $k$-WPC class must restrict to, or otherwise force, the critical equality. The most direct corrected formulation imposes precisely this necessary equality.

\begin{definition}[Criticality]\label{def:critical}
A hypergraph $H$ on $V$ is \emph{$k$-critical} if $\rho(H)=k(|V|-1)$.
\end{definition}

\begin{conjecture}[Critical form]\label{conj:critical}
Every $k$-critical, $k$-WPC hypergraph has a $k$-distinguishable tree assignment.
\end{conjecture}

Criticality is not merely an edge count verified after a decomposition has been found. Weak partition connectivity supplies $k$ edge-disjoint spanning trees in every assigned graph, and at criticality those trees exhaust all assigned edges. Thus every tree assignment of a critical $k$-WPC hypergraph has an ordered $k$-tree decomposition. The unresolved assertion in \Cref{conj:critical} is the existential statement that some tree assignment admits at least one decomposition whose signature fiber is a singleton.

Our positive result is based on the following certificate. It contains the extremal-layer star constructions as special cases and permits arbitrary assigned trees, including labels of arbitrary rank in interior layers.

\begin{definition}[Layer-contained certificate]\label{def:layer-contained}
A \emph{layer-contained certificate} for a hypergraph $H$ consists of a tree assignment $G=\biguplus_{e\in E(H)}F_e$, an ordered $k$-tree decomposition $\TT=(T_0,\ldots,T_{k-1})$ of $G$, and a map $\ell:E(H)\to\{0,\ldots,k-1\}$ such that $E(F_e)\subseteq E(T_{\ell(e)})$ for every hyperedge occurrence $e$.
\end{definition}

No star hypothesis is imposed, and labels assigned to interior layers may have arbitrary rank.

\begin{theorem}[Layer-contained positive instances]\label{thm:layer-positive}
If a hypergraph $H$ has a layer-contained certificate, then $H$ is $k$-critical and $k$-WPC, and the certified decomposition is the unique ordered $k$-tree decomposition having its signature. Hence the certified tree assignment is $k$-distinguishable.
\end{theorem}

The certificate is stable under the natural gluing operation.

\begin{theorem}[One-vertex sum closure]\label{thm:sum-closure}
Let $H^{(1)}$ and $H^{(2)}$ have layer-contained certificates for the same integer $k$. Suppose that their vertex sets meet in exactly one vertex $r$ and that their hyperedge-occurrence sets are disjoint. Then the one-vertex sum $H^{(1)}\vee_r H^{(2)}$ has a layer-contained certificate. In particular, it is $k$-critical, $k$-WPC, and has a $k$-distinguishable tree assignment.
\end{theorem}

The logical structure is summarized in \Cref{fig:logic}. The two right-hand nodes form a case distinction within the class satisfying the lower bound; neither is asserted without its displayed additional condition.

\begin{figure}[ht]
\centering
\begin{tikzpicture}[x=1mm,y=1mm,>=Latex,thick,
 every node/.style={draw,rounded corners,align=center,inner sep=4pt,font=\small}]
\node[text width=28mm] (weak) at (15,24) {$k$-weak partition\\connectivity};
\node[text width=28mm] (lower) at (55,24) {$\rho(H)\ge k(t-1)$};
\node[text width=34mm] (over) at (104,36) {$\rho(H)>k(t-1)$\\no full decomposition};
\node[text width=34mm] (critical) at (104,12) {$\rho(H)=k(t-1)$\\critical regime};
\node[text width=34mm] (cert) at (24,-10) {layer-contained certificate};
\node[text width=48mm] (positive) at (82,-10) {$k$-critical, $k$-WPC, and\\signature-unique};
\draw[->] (weak)--(lower);
\draw[->] (lower)--node[draw=none,above left,inner sep=1pt] {if strict} (over);
\draw[->] (lower)--node[draw=none,below left,inner sep=1pt] {if equality} (critical);
\draw[->] (cert)--(positive);
\end{tikzpicture}
\caption{The edge-count dichotomy and the positive certificate. Weak partition connectivity gives a lower bound, whereas a full ordered tree decomposition requires equality.}
\label{fig:logic}
\end{figure}

\section{The excess obstruction and the critical regime}\label{sec:excess}

This section establishes the excess obstruction and the exact packing statement available in the critical regime.

\begin{lemma}[Discrete-partition lower bound]\label{lem:discrete}
If $H$ is $k$-WPC on $t$ vertices, then $\rho(H)\ge k(t-1)$.
\end{lemma}

\begin{proof}
Apply \eqref{eq:kwpc} to the discrete partition $\PP_0$. Since $p_{\PP_0}(e)=|e|$ for every hyperedge occurrence $e$, one has $w_H(\PP_0)=\sum_e(|e|-1)=\rho(H)$ and $|\PP_0|-1=t-1$.
\end{proof}

\begin{lemma}[Tree-assignment and decomposition counts]\label{lem:counts}
Every tree assignment $G=\biguplus_{e}F_e$ of $H$ has $|E(G)|=\rho(H)$ edges. Every ordered $k$-tree decomposition on a $t$-vertex set has $k(t-1)$ edges.
\end{lemma}

\begin{proof}
Each $F_e$ is a tree on $|e|$ vertices and therefore has $|e|-1$ edges. Because the label classes are disjoint as edge occurrences, $|E(G)|=\sum_e|E(F_e)|=\sum_e(|e|-1)=\rho(H)$. If $E(G)=E(T_0)\dot\cup\cdots\dot\cup E(T_{k-1})$ and every $T_i$ is a spanning tree on $t$ vertices, then each $T_i$ has $t-1$ edges, so $|E(G)|=\sum_i|E(T_i)|=k(t-1)$.
\end{proof}

\begin{proof}[Proof of \Cref{thm:excess-obstruction}]
A $k$-distinguishable tree assignment has, by definition, an ordered $k$-tree decomposition. Applying \Cref{lem:counts} to the same labelled graph gives $\rho(H)=|E(G)|=k(t-1)$. Hence an overfull $k$-WPC hypergraph cannot have a $k$-distinguishable full tree assignment.
\end{proof}

\begin{proof}[Proof of \Cref{thm:full-edge-counterexamples}]
Let $\PP$ be a partition of $V$ with $s$ blocks. Every hyperedge occurrence of $H_{t,k,q}$ equals $V$ and therefore meets all $s$ blocks. Thus each occurrence contributes $s-1$, and $w_{H_{t,k,q}}(\PP)=(k+q)(s-1)\ge k(s-1)$. Hence $H_{t,k,q}$ is $k$-WPC\@. Its excess is $\rho(H_{t,k,q})=(k+q)(t-1)>k(t-1)$, so \Cref{thm:excess-obstruction} applies. When $t=2$ and $k=q=1$, the full hyperedge is the unique two-vertex edge, giving two labelled parallel copies.
\end{proof}

\begin{remark}[Full assignments versus distinguishable subgraphs]\label{rem:subgraph}
The obstruction concerns the literal full-assignment statement. A result that asks only for a $k$-distinguishable labelled subgraph of a tree assignment may delete assigned graph edges and is not contradicted by overfullness. The distinction is essential: a full tree assignment always has $\rho(H)$ edges, whereas a retained subgraph may have exactly $k(t-1)$ edges.
\end{remark}

\begin{example}[The minimal literal counterexample]\label{ex:minimal}
Let $V=\{a,b\}$, let $k=1$, and let $H$ consist of two labelled copies of $\{a,b\}$. The only nontrivial partition has weak excess $2\ge1$, so $H$ is $1$-WPC\@. Every full tree assignment is the two-edge labelled multigraph on $a,b$, while a one-tree decomposition on two vertices has exactly one edge. Thus the conjecture fails before signature uniqueness is considered.
\end{example}

The following crossing bound separates the existence of a tree packing from the requirement that the packing exhaust all assigned edges.

\begin{lemma}[Crossing inequality for an assigned tree]\label{lem:tree-cross-lower}
Let $F_e$ be a tree on a hyperedge $e$, and let $\PP$ be a partition of the ambient vertex set. Then $\cross_{F_e}(\PP)\ge p_{\PP}(e)-1$, where $\cross_{F_e}(\PP)$ is the number of edges of $F_e$ whose endpoints lie in different blocks of $\PP$.
\end{lemma}

\begin{proof}
Identify all vertices of $F_e$ that lie in the same block of $\PP$, and delete the resulting loops. The quotient is a connected multigraph whose vertices are precisely the $p_{\PP}(e)$ blocks met by $e$. Its nonloop edge occurrences are in bijection with the edges of $F_e$ crossing $\PP$. A connected multigraph on $p_{\PP}(e)$ vertices has at least $p_{\PP}(e)-1$ edges, which proves the claim.
\end{proof}

\begin{theorem}[Packing for every assignment; decomposition at equality]\label{thm:any-assignment-packing}
Let $H$ be $k$-WPC and let $G$ be any tree assignment of $H$. Then $G$ contains $k$ pairwise edge-disjoint spanning trees. If, in addition, $H$ is $k$-critical, those trees exhaust $E(G)$ and hence form a full ordered $k$-tree decomposition after an arbitrary ordering.
\end{theorem}

\begin{proof}
For every partition $\PP$ of $V$, summing \Cref{lem:tree-cross-lower} over the label classes gives
\begin{equation}\label{eq:graph-dominates-hypergraph}
 \cross_G(\PP)=\sum_{e\in E(H)}\cross_{F_e}(\PP)
 \ge\sum_{e\in E(H)}(p_{\PP}(e)-1)=w_H(\PP)\ge k(|\PP|-1).
\end{equation}
The Nash--Williams--Tutte theorem for multigraphs therefore yields $k$ pairwise edge-disjoint spanning trees in $G$~\cite{NW61,Tutte61}. Their union has exactly $k(t-1)$ edges. If $H$ is $k$-critical, then \Cref{lem:counts} gives $|E(G)|=\rho(H)=k(t-1)$, so every edge of $G$ belongs to the union. Ordering the resulting trees gives a full ordered decomposition.
\end{proof}

\begin{corollary}[The case $k=1$]\label{cor:k-one}
The critical conjecture is true for $k=1$.
\end{corollary}

\begin{proof}
Let $H$ be $1$-critical and $1$-WPC, and choose any tree assignment $G$. By \Cref{thm:any-assignment-packing}, $G$ contains a spanning tree. By criticality and \Cref{lem:counts}, $G$ has exactly $t-1$ edges, so that spanning tree is all of $G$. There is only one ordered one-tree decomposition, namely $(G)$, and its signature is therefore unique.
\end{proof}

Thus the critical conjecture is not an existence problem for tree packings. It asks whether there exist a tree assignment and an ordered decomposition whose signature fiber is a singleton.

\begin{proposition}[Exact surplus of the full-edge bundle]\label{prop:bundle-surplus}
If $\PP$ has $s$ blocks, then $w_{H_{t,k,q}}(\PP)-k(s-1)=q(s-1)$, while $\rho(H_{t,k,q})-k(t-1)=q(t-1)$. In particular, every nontrivial partition has positive weak-connectivity margin although the full tree assignment is overfull.
\end{proposition}

\begin{proof}
The first identity follows from $w_{H_{t,k,q}}(\PP)=(k+q)(s-1)$; the second is the specialization to the discrete partition.
\end{proof}

\begin{proposition}[Two-vertex classification]\label{prop:two-vertex-classification}
Let $L_m$ be the labelled multihypergraph on $\{x,y\}$ consisting of $m$ copies of $\{x,y\}$. For fixed $k\ge1$, the hypergraph $L_m$ is $k$-WPC if and only if $m\ge k$, and it has a $k$-distinguishable full tree assignment if and only if $m=k$.
\end{proposition}

\begin{proof}
The only nontrivial partition is $\{\{x\},\{y\}\}$, for which $w_{L_m}=m$. Hence $k$-WPC is equivalent to $m\ge k$. A full tree assignment consists of the same $m$ labelled parallel graph edges, whereas a $k$-tree decomposition on two vertices contains exactly one edge in each layer and therefore exactly $k$ edges. Thus a decomposition exists only if $m=k$. When $m=k$, order the labels as $e_0,\ldots,e_{k-1}$ and place the edge labelled $e_i$ in $T_i$. Its signature coordinate is $i$. Any decomposition with the same signature must place each labelled edge in the same layer, so the displayed decomposition is signature-unique.
\end{proof}

\begin{remark}[Sharpness of the numerical obstruction]\label{rem:sharp}
For a $k$-WPC hypergraph, the inequality $\rho(H)<k(t-1)$ is impossible by \Cref{lem:discrete}. Hence overfullness is the only failure detected by the discrete partition. After imposing $\rho(H)=k(t-1)$, \Cref{thm:any-assignment-packing} also removes the existence problem for a full decomposition. The unresolved implication is the following signature question: does every $k$-critical, $k$-WPC hypergraph admit a tree assignment with a signature-unique ordered decomposition?
\end{remark}

\section{Critical normalization and the intersection-matrix route}\label{sec:matrix}

This section makes precise the role of criticality in the intersection-matrix argument of~\cite{GLSTW24}. The equality $\rho(H)=k(t-1)$ is not needed merely to deduce weak partition connectivity from the weight inequalities. It is the condition under which a full tree assignment supplies exactly the $k(t-1)$ labelled bottom rows required to complement the cycle-space block to a square minor.

\subsection{Support hypergraphs and partition inequalities}

Let $I_1,\ldots,I_t\subseteq[n]$. For $s\in[n]$, define $e_s=\{j\in[t]:s\in I_j\}$. Retain the labelled occurrence $\varepsilon_s=(s,e_s)$ precisely when $|e_s|\ge2$, and define the labelled support hypergraph $H(I_1,\ldots,I_t)=([t],E_I)$, where $E_I=\{\varepsilon_s:s\in[n],\ |e_s|\ge2\}$. Equal supports with different indices remain distinct occurrences.

For nonempty $J\subseteq[t]$, set
\begin{equation}\label{eq:wt-def}
 \wt(I_J)\defeq\sum_{j\in J}|I_j|-\left|\bigcup_{j\in J}I_j\right|
 =\sum_{s=1}^n\max\{0,|e_s\cap J|-1\}.
\end{equation}
The second identity follows by collecting the contribution of each index $s$. We use the same definition, with $[n]$ replaced by an arbitrary finite ground set, whenever the sets under consideration are indexed by such a ground set. Supports of size at most one contribute zero, and $\wt(I_{[t]})=\rho(H(I_1,\ldots,I_t))$.

\begin{lemma}[Partition inequality from weights]\label{lem:weight-partition}
Let $H=H(I_1,\ldots,I_t)$. Suppose that $\wt(I_J)\le k(|J|-1)$ for every nonempty proper subset $J\subsetneq[t]$ and that $\wt(I_{[t]})\ge k(t-1)$. Then $H$ is $k$-WPC\@. If the total-weight inequality is an equality, then $H$ is also $k$-critical.
\end{lemma}

\begin{proof}
Let $\PP=\{P_1,\ldots,P_s\}$ be a partition of $[t]$. For every support $e$, if $r=p_{\PP}(e)$, then the nonempty sets $e\cap P_a$ partition $e$, and therefore $|e|-1=(r-1)+\sum_{a=1}^s\max\{0,|e\cap P_a|-1\}$. Summing this identity over the retained occurrences and applying \eqref{eq:wt-def} gives the exact formula
\begin{equation}\label{eq:weight-partition-exact}
 w_H(\PP)=\wt(I_{[t]})-\sum_{a=1}^s\wt(I_{P_a}).
\end{equation}
If $s=1$, both sides of the desired weak partition inequality are zero. If $s\ge2$, every $P_a$ is a nonempty proper subset of $[t]$, and hence $w_H(\PP)\ge k(t-1)-\sum_{a=1}^s k(|P_a|-1)=k(t-1)-k(t-s)=k(s-1)$. Thus $H$ is $k$-WPC\@. If $\wt(I_{[t]})=k(t-1)$, then $\rho(H)=\wt(I_{[t]})=k(t-1)$, which is criticality.
\end{proof}

The preceding implication is reversible at criticality under the canonical support representation.

\begin{proposition}[Exact critical correspondence]\label{prop:critical-weight-equivalence}
Let $H=([t],E)$ be a labelled multihypergraph whose occurrences are indexed by a finite label set $S$, and write $e_s$ for the support of label $s$. Define $I_j=\{s\in S:j\in e_s\}$ for $j\in[t]$, and compute weights in the finite ground set $S$. Then $\wt(I_J)=\sum_{s\in S}\max\{0,|e_s\cap J|-1\}$ for every nonempty $J\subseteq[t]$. Moreover, the following are equivalent:
\begin{enumerate}[label=\textup{(\roman*)},leftmargin=2.2em]
\item $H$ is $k$-critical and $k$-WPC;
\item $\wt(I_{[t]})=k(t-1)$ and $\wt(I_J)\le k(|J|-1)$ for every nonempty proper subset $J\subsetneq[t]$.
\end{enumerate}
\end{proposition}

\begin{proof}
The weight identity is the same label-by-label count as in \eqref{eq:wt-def}. After choosing a bijection from $S$ to $[|S|]$, \Cref{lem:weight-partition} shows that (ii) implies (i); the assertion is independent of the chosen bijection. For the converse, assume (i). The total equality is $\wt(I_{[t]})=\rho(H)=k(t-1)$. Fix a nonempty proper subset $J\subsetneq[t]$ and consider the partition $\PP_J=\{J\}\cup\{\{v\}:v\in[t]\setminus J\}$. It has $|\PP_J|-1=t-|J|$. For a label $s$, the blocks of $\PP_J$ met by $e_s$ consist of the block $J$ when $e_s\cap J\ne\varnothing$, together with one singleton block for each vertex of $e_s\setminus J$. Consequently, $(|e_s|-1)-(p_{\PP_J}(e_s)-1)=\max\{0,|e_s\cap J|-1\}$. Summing over $s$ yields $\rho(H)-w_H(\PP_J)=\wt(I_J)$. Since $H$ is $k$-WPC and $k$-critical, $\wt(I_J)=k(t-1)-w_H(\PP_J)\le k(t-1)-k(t-|J|)=k(|J|-1)$. Thus (ii) holds.
\end{proof}

\subsection{The cycle-space block and the selected determinant minor}

Let $K_t$ be the complete graph on $[t]$, let $m=\binom{t}{2}$, and let $r=m-(t-1)=\binom{t-1}{2}$. Fix a field $\FFq$, regard $x_1,\ldots,x_n$ as algebraically independent indeterminates over $\FFq$, and take all matrix ranks below over the rational function field $\FFq(x_1,\ldots,x_n)$. Orient each edge $\{u,v\}$ of $K_t$ from the smaller endpoint to the larger, and order the edges lexicographically. Let $D_t$ be the resulting oriented vertex--edge incidence matrix; its kernel is the cycle space of $K_t$ over $\FFq$.

Let $B_t$ be the $r\times m$ matrix used in~\cite{GLSTW24}: the row indexed by $1\le u<v\le t-1$ is the circulation vector of the triangle $u,v,t$, with signs taken relative to the fixed edge orientations. Each such row belongs to $\ker D_t$. The rows are linearly independent because the row indexed by $(u,v)$ is the only one having a nonzero coordinate on the non-star edge $\{u,v\}$. Moreover, $D_t$ has rank $t-1$ over every field. Indeed, the sum of its vertex rows is zero, so its rank is at most $t-1$; after deleting the row indexed by $t$, the columns indexed by $\{1,t\},\ldots,\{t-1,t\}$ form a nonsingular diagonal submatrix up to signs, so its rank is at least $t-1$. Rank--nullity therefore gives $\dim\ker D_t=m-(t-1)=r$. Hence the displayed triangle circulations form a basis of the cycle space. When $t=2$, $B_t$ is the empty $0\times1$ matrix.

\begin{lemma}[Cycle-matrix deletion criterion]\label{lem:cycle-deletion}
For $X\subseteq E(K_t)$, the restricted matrix $B_t[E(K_t)\setminus X]$ has full row rank $r$ if and only if $X$ is a forest. In particular, if $|X|=t-1$, the square matrix $B_t[E(K_t)\setminus X]$ is nonsingular if and only if $X$ is a spanning tree.
\end{lemma}

\begin{proof}
The restriction fails to have full row rank precisely when a nonzero linear combination of the rows of $B_t$ vanishes on every column outside $X$. Because the rows form a basis of the cycle space, this is equivalent to the existence of a nonzero cycle-space vector supported in $X$.

If $X$ contains a cycle $C$, orient $C$ cyclically and assign coefficient $1$ or $-1$ to each edge according to whether its fixed orientation agrees with the cyclic orientation. The resulting nonzero vector lies in the kernel of the oriented incidence matrix and is supported on $C\subseteq X$. Hence the restricted rows are dependent. In characteristic $2$, the same construction assigns coefficient $1$ to every edge of $C$; the incidence equation at each vertex still vanishes because exactly two edges of $C$ are incident with that vertex.

Conversely, suppose that a nonzero cycle-space vector $z$ is supported in $X$. If the support graph of $z$ were a forest, it would contain a leaf $v$ incident with a unique support edge $f$. The incidence equation at $v$ would then read $\pm z_f=0$, contradicting $f\in\supp(z)$. Thus the support of every nonzero cycle-space vector contains a cycle, so $X$ cannot be a forest. This proves the first assertion. If $|X|=t-1$, a forest on $t$ vertices has $t-1$ edges exactly when it is a spanning tree; in that case the restricted matrix has $m-(t-1)=r$ columns and is square.
\end{proof}

For a finite set $S\subseteq[n]$, let $V_k(S)$ be the $|S|\times k$ Vandermonde-type matrix whose row indexed by $s\in S$ is $(1,x_s,\ldots,x_s^{k-1})$. With columns ordered first by an edge $a\in E(K_t)$ and then by an exponent $i\in\{0,\ldots,k-1\}$, the intersection matrix of~\cite{GLSTW24} is
\begin{equation}\label{eq:full-intersection-matrix}
 \mathcal M_{k,I}(x)=
 \begin{pmatrix}
 B_t\otimes I_k\\[1mm]
 \diag\bigl(V_k(I_u\cap I_v):\{u,v\}\in E(K_t)\bigr)
 \end{pmatrix}.
\end{equation}
The block order in the lower diagonal is fixed once and for all. Following the terminology of~\cite{GLSTW24}, this generally rectangular variable matrix is called \emph{nonsingular} when it has full column rank over $\FFq(x_1,\ldots,x_n)$, equivalently when it contains a $km\times km$ row-selected minor whose determinant is a nonzero polynomial. We use the standard term \emph{full column rank} below.

To obtain the layer-major form needed below, both rows and columns must be reordered. Index the top rows of $B_t\otimes I_k$ by pairs $(\gamma,i)$, where $\gamma$ is a row of $B_t$, and index its columns by $(a,j)$. Reorder the top rows as $(i,\gamma)$ and all columns as $(j,a)$. The entry in the reordered top block is then $\delta_{ij}(B_t)_{\gamma a}$, so the block becomes $I_k\otimes B_t$. Equivalently, there are perfect-shuffle permutation matrices $P_{\mathrm{row}}$ and $P_{\mathrm{col}}$ satisfying $P_{\mathrm{row}}(B_t\otimes I_k)P_{\mathrm{col}}=I_k\otimes B_t$.
Applying the block row permutation $\diag(P_{\mathrm{row}},I)$ and the column permutation $P_{\mathrm{col}}$ to \eqref{eq:full-intersection-matrix} preserves rank. More precisely, these permutations induce a bijection between square minors of the original and reordered matrices, and corresponding determinants differ only by the appropriate row and column permutation signs. In particular, the existence of a nonzero full-column-size minor is preserved. In the reordered lower block, the row indexed by a pair $a=\{u,v\}$ and an element $s\in I_u\cap I_v$ has entry $x_s^i$ in column $(i,a)$ and zero elsewhere.

Let $G$ be a tree assignment of $H(I_1,\ldots,I_t)$ with exactly $k(t-1)$ labelled graph-edge occurrences. For $f\in E(G)$, let $a(f)\in E(K_t)$ be its unordered endpoint pair and let $s(f)\in[n]$ be the index labelling its hyperedge occurrence. Define a square matrix $M_G(x)$ with $km$ columns indexed by $(i,a)\in\{0,\ldots,k-1\}\times E(K_t)$. Its top $kr$ rows form $I_k\otimes B_t$. For every $f\in E(G)$, let $r_f=(a(f),s(f))$ be the corresponding lower-row index in the reordered intersection matrix, and let $R_f$ denote the row at that index; its entry in column $(i,a)$ is $x_{s(f)}^i$ when $a=a(f)$ and is zero otherwise. The row indices $r_f$ are pairwise distinct: equal labels force the edges into the same simple assigned tree $F_{e_s}$, which contains at most one edge with any prescribed endpoint pair, while different labels give different second coordinates. The number of selected bottom-row occurrences is therefore $k(t-1)$, and the total number of rows is $kr+k(t-1)=km$. Thus $M_G(x)$ is a square row-selected submatrix of a row-and-column permutation of $\mathcal M_{k,I}(x)$.

For a labelled subgraph $Q$ of $G$, write $\underline Q=\{a(f):f\in E(Q)\}\subseteq E(K_t)$ for its underlying simple edge set. If $Q$ is a tree, then no two of its labelled edges have the same endpoint pair, and therefore $|\underline Q|=|E(Q)|$.

\begin{lemma}[Decomposition-indexed generalized-Laplace expansion]\label{lem:det-expansion}
Let $\mathfrak D_k(G)$ denote the set of ordered $k$-tree decompositions of $G$, and let $L_s=\{f\in E(G):s(f)=s\}$ be the $s$th label class. For $\QQ=(Q_0,\ldots,Q_{k-1})\in\mathfrak D_k(G)$, define $\sigma_s(\QQ)=\sum_{i=0}^{k-1}i|E(Q_i)\cap L_s|$ and $x^{\sigma(\QQ)}=\prod_{s=1}^n x_s^{\sigma_s(\QQ)}$. Then, relative to the fixed row and column orders, there is a sign $\epsilon(\QQ)\in\FFq^{\times}$, obtained from an integer sign $\pm1$, such that
\begin{equation}\label{eq:det-expansion}
 \det M_G(x)=\sum_{\QQ\in\mathfrak D_k(G)}c_G(\QQ)x^{\sigma(\QQ)},
 \qquad
 c_G(\QQ)=\epsilon(\QQ)\prod_{i=0}^{k-1}\det B_t[E(K_t)\setminus\underline{Q_i}]\in\FFq^{\times}.
\end{equation}
The sum in \eqref{eq:det-expansion} is indexed by decompositions; equal monomials have not yet been combined. Consequently, for every $\alpha\in\ZZ_{\ge0}^n$, write $x^\alpha=\prod_{s=1}^n x_s^{\alpha_s}$ and let $[x^\alpha]$ denote coefficient extraction; then
\begin{equation}\label{eq:collected-coefficient}
 [x^\alpha]\det M_G(x)=
 \sum_{\substack{\QQ\in\mathfrak D_k(G)\\ \sigma(\QQ)=\alpha}}c_G(\QQ).
\end{equation}
Thus every decomposition-level scalar $c_G(\QQ)$ is nonzero over the ground field, including in characteristic $2$, whereas a collected coefficient in \eqref{eq:collected-coefficient} may vanish. The determinant of the empty $0\times0$ matrix, which occurs when $t=2$, is interpreted as $1$.
\end{lemma}

\begin{proof}
Let $\mathcal R=\{r_f:f\in E(G)\}$ be the selected bottom-row set and put $p=|\mathcal R|=k(t-1)$. Apply the generalized Laplace formula along the rows in $\mathcal R$. For each $p$-element column set $C$, the corresponding term is, up to the Laplace sign, $\det M_G[\mathcal R,C]\det M_G[\mathcal R^c,C^c]$. Expand only the bottom determinant by the Leibniz formula and leave the complementary top determinant unexpanded.

A nonzero bottom Leibniz term chooses, for every $f\in E(G)$, one column in the support of $R_f$. Since $R_f$ is supported precisely in the columns $(i,a(f))$, the choice has the form $(\lambda(f),a(f))$ for a unique map $\lambda:E(G)\to\{0,\ldots,k-1\}$. The chosen columns are distinct if and only if no two edges with the same endpoint pair receive the same value of $\lambda$. Conversely, every map $\lambda$ satisfying this condition determines a unique selected column set $C$ and a unique bottom Leibniz term. For $0\le i\le k-1$, let $Q_i=\lambda^{-1}(i)$. Then the $Q_i$ partition $E(G)$ and $|E(Q_i)|=|\underline{Q_i}|$ for every $i$.

Deleting the selected columns from the top block $I_k\otimes B_t$ leaves a block-diagonal matrix whose $i$th block is $B_t[E(K_t)\setminus\underline{Q_i}]$. The complementary top matrix has $kr$ rows and $kr$ columns. If its determinant is nonzero, then its rank is $kr$; because the rank of a block-diagonal matrix is the sum of the ranks of its blocks, every block must have full row rank $r$. By \Cref{lem:cycle-deletion}, each $\underline{Q_i}$ is a forest. Hence $|E(Q_i)|=|\underline{Q_i}|\le t-1$. Since the $Q_i$ partition the $k(t-1)$ edges of $G$, equality must hold for every $i$. Thus each $\underline{Q_i}$ is a forest with $t-1$ edges on $t$ vertices and therefore a spanning tree. Because the edges in $Q_i$ have distinct endpoint pairs, $Q_i$ itself is a spanning tree of the labelled multigraph. Hence every nonzero decomposition-level Laplace term determines an element of $\mathfrak D_k(G)$.

Conversely, fix $\QQ=(Q_0,\ldots,Q_{k-1})\in\mathfrak D_k(G)$. For each $f\in Q_i$, select from $R_f$ the entry in column $(i,a(f))$. A spanning tree contains no parallel pair, so the selected columns are distinct within each layer; the fact that the $Q_i$ partition $E(G)$ makes this a valid bottom transversal. It is uniquely determined by $\QQ$. By \Cref{lem:cycle-deletion}, every complementary block $B_t[E(K_t)\setminus\underline{Q_i}]$ is square and nonsingular. Therefore the associated decomposition-level term is nonzero.

The selected bottom entries contribute $\prod_{i=0}^{k-1}\prod_{f\in E(Q_i)}x_{s(f)}^i=x^{\sigma(\QQ)}$. The product of the Laplace sign, the bottom permutation sign, and the complementary block determinants is precisely $c_G(\QQ)$. This proves \eqref{eq:det-expansion} and the one-to-one correspondence between ordered decompositions and decomposition-level terms. Collecting terms with the same exponent vector gives \eqref{eq:collected-coefficient}.
\end{proof}

\begin{remark}[Cancellation inside a signature fiber]\label{rem:coefficient-cancellation}
The distinction between a decomposition-level scalar and a collected polynomial coefficient is essential. Take $t=3$ and $k=2$, and let two distinct labels $a$ and $b$ both have support $\{1,2,3\}$. Assign to each label the path with edge set $\{\{1,2\},\{2,3\}\}$. With $B_3=(1,-1,1)$, columns ordered as $(0,12),(0,13),(0,23),(1,12),(1,13),(1,23)$, and selected bottom rows ordered as $a12,a23,b12,b23$, the matrix is
\begin{equation}\label{eq:cancellation-example}
 M_G(x)=
 \begin{pmatrix}
 1&-1&1&0&0&0\\
 0&0&0&1&-1&1\\
 1&0&0&x_a&0&0\\
 0&0&1&0&0&x_a\\
 1&0&0&x_b&0&0\\
 0&0&1&0&0&x_b
 \end{pmatrix},
 \qquad \det M_G(x)=(x_a-x_b)^2.
\end{equation}
Indeed, subtracting the third and fourth rows from the fifth and sixth rows, respectively, and then expanding along the last two rows gives the displayed determinant. There are four ordered two-tree decompositions. Two keep each label class intact and contribute the monomials $x_a^2$ and $x_b^2$. The other two split both label classes between the layers; they have the same signature and both contribute the monomial $x_ax_b$. Their nonzero decomposition-level scalars sum to the collected coefficient $-2$, which vanishes in characteristic $2$. Thus \eqref{eq:det-expansion} must be read before like monomials are combined, and \eqref{eq:collected-coefficient} is the correct statement about polynomial coefficients.
\end{remark}

\begin{theorem}[Critical form implies intersection-matrix full column rank]\label{thm:critical-implies-matrix}
Assume the critical form stated in \Cref{conj:critical}. Let $t\ge2$, $k\ge1$, let $\FFq$ be any field, and let $I_1,\ldots,I_t\subseteq[n]$ satisfy $\wt(I_J)\le k(|J|-1)$ for every nonempty proper subset $J\subsetneq[t]$ and $\wt(I_{[t]})=k(t-1)$. Then $\mathcal M_{k,I}(x)$ has full column rank over $\FFq(x_1,\ldots,x_n)$. In particular, the intersection-matrix conjecture, Conjecture~1.5 of~\cite{GLSTW24}, follows from the critical hypergraph conjecture.
\end{theorem}

\begin{proof}
By \Cref{lem:weight-partition}, the support hypergraph $H=H(I_1,\ldots,I_t)$ is $k$-critical and $k$-WPC\@. By the assumed critical conjecture, choose a tree assignment $G$ with a signature-unique ordered decomposition $\TT=(T_0,\ldots,T_{k-1})$. By \Cref{lem:counts}, $G$ has exactly $k(t-1)$ labelled edge occurrences, so the square row-selected minor $M_G(x)$ is defined over $\FFq[x_1,\ldots,x_n]$.

For an ordered decomposition $\QQ$, the exponent $\sigma_s(\QQ)$ from \Cref{lem:det-expansion} is the signature coordinate of the occurrence labelled by $s$ whenever $|e_s|\ge2$; when $|e_s|\le1$, the class $L_s$ is empty and $\sigma_s(\QQ)=0$. Since $x_1,\ldots,x_n$ are algebraically independent, two monomials $x^{\sigma(\QQ)}$ are equal exactly when the corresponding signatures agree on every retained occurrence. The signature fiber of $\TT$ is a singleton, so the coefficient formula \eqref{eq:collected-coefficient} gives $[x^{\sigma(\TT)}]\det M_G(x)=c_G(\TT)\ne0$. Therefore $\det M_G(x)$ is a nonzero polynomial.

The row-and-column reordered intersection matrix contains $M_G(x)$ as a square row-selected submatrix, so it has full column rank. Row and column permutations preserve rank, and therefore the original matrix \eqref{eq:full-intersection-matrix} has full column rank as required.
\end{proof}

\begin{proposition}[Criticality and the selected full-assignment row set]\label{prop:square}
For a support hypergraph $H=H(I_1,\ldots,I_t)$, the following are equivalent:
\begin{enumerate}[label=\textup{(\roman*)},leftmargin=2.2em]
\item $\rho(H)=k(t-1)$;
\item $\wt(I_{[t]})=k(t-1)$;
\item every full tree assignment $G$ of $H$ has exactly $k(t-1)$ labelled graph-edge occurrences and therefore yields, by the construction above, a square selected matrix $M_G(x)$.
\end{enumerate}
\end{proposition}

\begin{proof}
The equivalence of (i) and (ii) is the identity $\rho(H)=\wt(I_{[t]})$. By \Cref{lem:counts}, every full tree assignment has exactly $\rho(H)$ labelled graph-edge occurrences, and the construction of $M_G(x)$ retains one bottom row for each such occurrence. Thus (i) and (iii) are equivalent.
\end{proof}

\begin{remark}[What is and is not equivalent]\label{rem:matrix-strength}
By \Cref{prop:critical-weight-equivalence}, the critical $k$-WPC input conditions are exactly the canonical support-hypergraph form of the intersection weight inequalities. The conclusion of \Cref{conj:critical}, however, is a sufficient structural certificate for matrix full column rank; no converse from matrix full column rank to the existence of a signature-singleton full tree assignment is established here. A nonzero determinant may, in principle, arise from several decomposition-level terms in the same signature fiber whose collected coefficient is nonzero.
\end{remark}

\begin{remark}[Overfull support hypergraphs]\label{rem:literal-matrix}
If the total condition were replaced by $\wt(I_{[t]})>k(t-1)$ while the proper-subset bounds were retained, the calculation in \Cref{lem:weight-partition} would still yield $k$-WPC\@. A full tree assignment would then supply more than $k(t-1)$ bottom rows. One could seek a square row-selected minor by deleting assigned graph-edge rows, but that is an additional pruning problem rather than a full-assignment implication.
\end{remark}

\section{Layer-contained certificates and explicit families}\label{sec:certificates}

The proof of \Cref{thm:layer-positive} uses two distinct consequences of layer containment: a quotient-rank bound in each layer and a layer-by-layer capacity argument for signature uniqueness.

\begin{lemma}[Layer quotient rank]\label{lem:layer-rank}
Let $T$ be a spanning tree of $V$, let $\PP$ be a partition of $V$ with $s=|\PP|$, and suppose that $E(T)$ is partitioned into classes $A_e$ indexed by a family $\mathcal L$ of hyperedges, with both endpoints of every edge in $A_e$ belonging to $e$. Then $s-1\le\sum_{e\in\mathcal L}(p_{\PP}(e)-1)$.
\end{lemma}

\begin{proof}
Contract every block of $\PP$ in $T$ and delete loops. The resulting quotient multigraph $Q$ is connected on the $s$ blocks, so its graphic-matroid rank is $r_Q(E(Q))=s-1$. For each $e\in\mathcal L$, let $Q_e$ be the set of nonloop quotient edges arising from $A_e$. The sets $Q_e$ partition $E(Q)$. Repeated use of the matroid rank inequality $r(X\cup Y)\le r(X)+r(Y)$ gives $s-1=r_Q(E(Q))\le\sum_{e\in\mathcal L}r_Q(Q_e)$. Every edge of $Q_e$ has both endpoints among the blocks meeting $e$, so $Q_e$ is supported on at most $p_{\PP}(e)$ quotient vertices. Its graphic rank is therefore at most $p_{\PP}(e)-1$. Combining these inequalities yields the asserted bound.
\end{proof}

\begin{proof}[Proof of \Cref{thm:layer-positive}]
Let $G=\biguplus_{e}F_e$, $\TT=(T_0,\ldots,T_{k-1})$, and $\ell$ be a layer-contained certificate. Because $G$ has a full ordered $k$-tree decomposition, \Cref{lem:counts} gives $\rho(H)=|E(G)|=k(t-1)$; hence $H$ is $k$-critical.

For $0\le i\le k-1$, let $E_i=\ell^{-1}(i)$. The sets $E_i$ partition $E(H)$. Layer containment gives $\dot\bigcup_{e\in E_i}E(F_e)\subseteq E(T_i)$. Conversely, if an edge occurrence $f\in E(T_i)$ has label $e$, then $f\in E(F_e)\subseteq E(T_{\ell(e)})$; because the layers form an edge partition, necessarily $\ell(e)=i$. Therefore $E(T_i)=\dot\bigcup_{e\in E_i}E(F_e)$. Fix a partition $\PP$ of $V$ with $s$ blocks. Apply \Cref{lem:layer-rank} to $T_i$, with classes $A_e=E(F_e)$ for $e\in E_i$. For every $i$ this gives $s-1\le\sum_{e\in E_i}(p_{\PP}(e)-1)$. Summing over all layers and using that the $E_i$ partition $E(H)$ yields $k(s-1)\le\sum_{i=0}^{k-1}\sum_{e\in E_i}(p_{\PP}(e)-1)=w_H(\PP)$. Thus $H$ is $k$-WPC\@.

It remains to prove signature uniqueness. For each $i$, put $A_i=\dot\bigcup_{e\in E_i}E(F_e)=E(T_i)$, so $|A_i|=t-1$. Let $\TT'=(T'_0,\ldots,T'_{k-1})$ be an ordered $k$-tree decomposition of $G$ satisfying $\sig_{\TT'}=\sig_{\TT}$. We prove by induction on $c$ that $T'_c=T_c$.

For $c=0$, every label $e\in E_0$ has target signature $0$. All layer indices are nonnegative, so equality of signatures implies that every edge of $F_e$ lies in $T'_0$. Hence $A_0\subseteq E(T'_0)$. Both sets have $t-1$ edges, and therefore $E(T'_0)=A_0=E(T_0)$. This also covers $k=1$.

Now fix $1\le c\le k-1$ and assume that $T'_j=T_j$ for all $j<c$. Let $e\in E_c$ and set $n_e=|E(F_e)|=|e|-1$. Because $E(F_e)\subseteq A_c$ and the certified layers are edge-disjoint, the induction hypothesis implies that no edge of $F_e$ lies in any $T'_j$ with $j<c$. Define $a_i(e)=|E(F_e)\cap E(T'_i)|$. Then $a_i(e)=0$ for $i<c$ and $\sum_{i=c}^{k-1}a_i(e)=n_e$. Under the certified decomposition, every edge of $F_e$ lies in layer $c$, so $\sig_{\TT}(e)=cn_e$. Equality of signatures gives
\begin{equation}\label{eq:induction-signature}
 0=\sum_{i=c}^{k-1}(i-c)a_i(e)=\sum_{i=c+1}^{k-1}(i-c)a_i(e).
\end{equation}
Every summand in the final expression is nonnegative. Hence $a_i(e)=0$ for all $i>c$, and consequently $a_c(e)=n_e$. Thus $E(F_e)\subseteq E(T'_c)$ for every $e\in E_c$, so $A_c\subseteq E(T'_c)$. Again both sets have $t-1$ edges; therefore $E(T'_c)=A_c=E(T_c)$. The induction is complete, and $\TT'=\TT$.
\end{proof}

\begin{remark}[Why no interior rank restriction is needed]\label{rem:no-interior-restriction}
For a single label with $n\ge2$ edges, the equation $\sum_i ia_i=cn$ does not determine the vector $(a_0,\ldots,a_{k-1})$ when $c$ is an interior index. That local ambiguity does not produce a second global decomposition. In the proof above, layers $0,\ldots,c-1$ are first fixed and filled to capacity; the signature equality then prevents the label from moving to any higher layer, and the capacity $|E(T'_c)|=t-1$ fixes the entire layer. Thus hyperedges of arbitrary rank may occur in every layer.
\end{remark}

\subsection{Stars and basic certified blocks}

For $c\in e$, let $S_e(c)$ denote the star on $e$ centered at $c$, with edge set $\{\{c,u\}:u\in e\setminus\{c\}\}$. Stars give convenient explicit assignments, although \Cref{thm:layer-positive} uses only layer containment.

\begin{proposition}[Critical full-edge bundles]\label{prop:critical-full-bundle}
Let $H$ consist of exactly $k$ labelled copies $e_0,\ldots,e_{k-1}$ of the full hyperedge $V$, where $|V|=t\ge2$. For each $i$, choose an arbitrary spanning tree $F_{e_i}$ on $V$ and set $T_i=F_{e_i}$. Then $(G,(T_0,\ldots,T_{k-1}),\ell)$, with $\ell(e_i)=i$, is a layer-contained certificate. Consequently, $H$ is $k$-critical, $k$-WPC, and has a $k$-distinguishable tree assignment.
\end{proposition}

\begin{proof}
The assigned trees are disjoint as labelled edge occurrences, even when their underlying endpoint pairs coincide, and their ordered union is a full $k$-tree decomposition. Each label class is contained in its prescribed layer, so \Cref{thm:layer-positive} applies. Directly, every partition with $s$ blocks has $w_H(\PP)=k(s-1)$ and $\rho(H)=k(t-1)$.
\end{proof}

\begin{example}[The rank-two specialization]\label{ex:rank-two}
Suppose that $H$ is an ordinary labelled multigraph with exactly $k(t-1)$ edges and satisfies the Nash--Williams--Tutte partition inequalities. Its tree assignment is forced: each rank-two hyperedge is replaced by its unique graph edge. The Nash--Williams--Tutte theorem gives $k$ edge-disjoint spanning trees, and the edge count forces them to exhaust $E(H)$. Since each rank-two hyperedge occurrence supplies its own label and its label class consists of a single edge, the corresponding signature coordinate is exactly its layer index. Hence every ordered tree decomposition is signature-unique. Thus the critical hypergraph conjecture reduces to ordinary tree packing in rank two.
\end{example}

For $k\ge3$ and $t\ge3$, \Cref{prop:critical-full-bundle} places a rank-$t$ label in every interior layer. It therefore gives an explicit family showing that an interior rank-two condition is unnecessary for layer-contained certificates.

\begin{definition}[Saturated two-sided star block]\label{def:saturated-block}
Fix $t\ge2$ and $k\ge2$, write $V=\{r,u_1,\ldots,u_{t-1}\}$, and choose centers $c_-,c_+\in V$. A saturated two-sided star block consists of two distinct labelled occurrences $e^-,e^+$ of the full hyperedge $V$, assigned respectively to the stars $S_V(c_-)$ and $S_V(c_+)$, together with $k-2$ ordinary labelled spanning trees $R_1,\ldots,R_{k-2}$, every edge of each $R_i$ being represented by its own rank-two hyperedge. The prescribed layers are $T_0=S_V(c_-)$, $T_i=R_i$ for $1\le i\le k-2$, and $T_{k-1}=S_V(c_+)$; the middle list is empty when $k=2$.
\end{definition}

\begin{proposition}[Saturated star blocks are certified]\label{prop:saturated-certified}
Every saturated two-sided star block has a layer-contained certificate and hence is $k$-critical, $k$-WPC, and admits a $k$-distinguishable tree assignment.
\end{proposition}

\begin{proof}
Each prescribed layer is a spanning tree. Assign $e^-$ to layer $0$, assign $e^+$ to layer $k-1$, and assign every rank-two label belonging to $R_i$ to layer $i$. Each label class is contained in its prescribed layer. Hence the data form a layer-contained certificate. The excess is $2(t-1)+(k-2)(t-1)=k(t-1)$, in agreement with \Cref{thm:layer-positive}.
\end{proof}

\begin{example}[A small genuinely hypergraphic certificate]\label{ex:nongraphic}
Let $k=2$ and $V=\{r,a,b\}$. Take one triple edge $e=\{r,a,b\}$ and two rank-two edges $f_a=\{r,a\}$ and $f_b=\{r,b\}$. Assign $F_e=\{ra,rb\}$ and place it in $T_0$, while the unique assigned edges of $f_a$ and $f_b$ form $T_1$. Both layers are spanning trees, and every label class is layer-contained. Thus \Cref{thm:layer-positive} applies. Directly, $\rho(H)=2+1+1=4=2(|V|-1)$. The discrete partition has weak excess $4$; the bipartition $\{r\mid ab\}$ has weak excess $3$; and each of the other two nontrivial bipartitions has weak excess $2$.
\end{example}

\subsection{One-vertex sums}

Let $H^{(1)}$ and $H^{(2)}$ have vertex sets $V_1$ and $V_2$ with $V_1\cap V_2=\{r\}$ and disjoint hyperedge-occurrence sets. Their one-vertex sum has vertex set $V_1\cup V_2$ and the disjoint union of the hyperedge occurrences.

\begin{lemma}[Trees glue at one vertex]\label{lem:tree-sum}
If $T^{(1)}$ and $T^{(2)}$ are spanning trees on $V_1$ and $V_2$, respectively, and $V_1\cap V_2=\{r\}$, then $T^{(1)}\cup T^{(2)}$ is a spanning tree on $V_1\cup V_2$.
\end{lemma}

\begin{proof}
The union is connected because both trees contain $r$. It has $(|V_1|-1)+(|V_2|-1)=|V_1\cup V_2|-1$ edges. A connected graph on $|V_1\cup V_2|$ vertices with this many edges is a tree. Equivalently, a cycle cannot pass from one side to the other and return because the two sides meet only at $r$; any cycle would lie entirely in one input tree.
\end{proof}

\begin{proof}[Proof of \Cref{thm:sum-closure}]
For $a\in\{1,2\}$, let $G^{(a)}=\biguplus_{e\in E(H^{(a)})}F^{(a)}_e$, let $\TT^{(a)}=(T^{(a)}_0,\ldots,T^{(a)}_{k-1})$, and let $\ell_a$ be a layer-contained certificate. Take the disjoint union of the two tree assignments. For each $i$, define $T_i=T_i^{(1)}\cup T_i^{(2)}$. By \Cref{lem:tree-sum}, $T_i$ is a spanning tree of $V_1\cup V_2$. The trees $T_i$ are pairwise edge-disjoint and cover the entire assigned graph because these properties hold separately on each side.

Define $\ell(e)=\ell_a(e)$ for $e\in E(H^{(a)})$. Every assigned tree $F_e^{(a)}$ remains contained in $T_{\ell_a(e)}^{(a)}\subseteq T_{\ell(e)}$. Hence the glued data form a layer-contained certificate, and \Cref{thm:layer-positive} gives all stated conclusions.
\end{proof}

The argument iterates for any finite family of certified blocks that admits an ordering $H^{(1)},\ldots,H^{(m)}$ such that, for every $j\ge2$, the vertex set of $H^{(j)}$ meets the union of the preceding vertex sets in exactly one vertex and all hyperedge-occurrence sets are pairwise disjoint. In particular, iterated one-vertex sums of the blocks in \Cref{prop:critical-full-bundle,prop:saturated-certified} and \Cref{ex:nongraphic} produce infinite families with hyperedges of unbounded rank.

\begin{figure}[ht]
\centering
\begin{tikzpicture}[scale=0.95,every node/.style={circle,draw,inner sep=1.2pt},every path/.style={line width=0.9pt}]
\node (r) at (0,0) {$r$};
\node (a1) at (-2,1.2) {$a_1$};
\node (a2) at (-2,-1.2) {$a_2$};
\node (a3) at (-3,0) {$a_3$};
\node (b1) at (2,1.2) {$b_1$};
\node (b2) at (2,-1.2) {$b_2$};
\node (b3) at (3,0) {$b_3$};
\draw (r)--(a1); \draw (r)--(a2); \draw (r)--(a3);
\draw[dashed] (r)--(b1); \draw[dashed] (r)--(b2); \draw[dashed] (r)--(b3);
\draw[decorate,decoration={brace,amplitude=5pt},thick] (-3.2,-1.55)--(-0.15,-1.55) node[midway,below=6pt,draw=none,rectangle] {left certificate};
\draw[decorate,decoration={brace,amplitude=5pt},thick] (0.15,-1.55)--(3.2,-1.55) node[midway,below=6pt,draw=none,rectangle] {right certificate};
\end{tikzpicture}
\caption{One-vertex gluing at $r$. Corresponding tree layers are united, and each resulting layer is again a spanning tree.}
\label{fig:sum}
\end{figure}

\section{Signature fibers, partition slack, and local obstructions}\label{sec:slack}

This section separates local signature arithmetic from the global capacity constraints used in \Cref{thm:layer-positive}, and records the exact relation between hypergraph and graph partition slack.

\subsection{The labelled matroid fiber}

Let $G=\biguplus_{e\in E(H)}F_e$ be a connected tree assignment of $H$. The bases of its graphic matroid $M(G)$ are precisely the spanning trees of $G$. Hence an ordered $k$-tree decomposition is an ordered basis partition $E(G)=B_0\dot\cup\cdots\dot\cup B_{k-1}$ with $B_i\in\mathcal B(M(G))$. Evaluated on the indicator vectors of these ordered basis partitions, the signature is the restriction of the linear map
\begin{equation}\label{eq:linear-sig}
 \Phi:(\ZZ^{E(G)})^k\longrightarrow\ZZ^{E(H)},\qquad
 \Phi(z_0,\ldots,z_{k-1})_e
 =\sum_{i=0}^{k-1}i\sum_{f\in E(F_e)}z_i(f).
\end{equation}
Distinguishability asks for a singleton fiber of $\Phi$ on the finite set of ordered basis partitions.

For a fixed label $e$, put $n_e=|E(F_e)|$ and $a_i(e)=|B_i\cap E(F_e)|$. Then $a_i(e)\in\ZZ_{\ge0}$, $\sum_{i=0}^{k-1}a_i(e)=n_e$, and the corresponding signature coordinate is $\sum_{i=0}^{k-1}i\,a_i(e)$. The extremal coordinates are rigid: $\sig(e)=0$ if and only if $(a_0,\ldots,a_{k-1})=(n_e,0,\ldots,0)$, and $\sig(e)=(k-1)n_e$ if and only if $(a_0,\ldots,a_{k-1})=(0,\ldots,0,n_e)$.

Interior averages are not rigid at the level of a single label.

\begin{proposition}[Local ambiguity of an interior signature]\label{prop:local-ambiguity}
Fix $k\ge3$, $1\le c\le k-2$, and $n\ge2$. There are distinct vectors $a,b\in\ZZ_{\ge0}^k$ such that $\sum_{i=0}^{k-1}a_i=\sum_{i=0}^{k-1}b_i=n$ and $\sum_{i=0}^{k-1}i\,a_i=\sum_{i=0}^{k-1}i\,b_i=cn$. One may take $b_c=n$, $a_{c-1}=a_{c+1}=1$, and $a_c=n-2$, with every unspecified coordinate of both $a$ and $b$ equal to zero.
\end{proposition}

\begin{proof}
Both vectors have coordinate sum $n$. The weighted sum of $a$ is $(c-1)+c(n-2)+(c+1)=cn$, equal to the weighted sum of $b$. They are distinct because $a_{c-1}=1$ while $b_{c-1}=0$.
\end{proof}

\begin{remark}[Local ambiguity is not global nonuniqueness]\label{rem:local-not-global}
\Cref{prop:local-ambiguity} shows only that one signature coordinate does not determine the distribution of one label in isolation. It does not construct a second ordered basis partition. For example, take $k=3$, $t=3$, and three labelled full hyperedges, one assigned wholly to each layer as in \Cref{prop:critical-full-bundle}. The middle label has two edges and target signature $2$; the alternative count vector $(1,0,1)$ has the same weighted sum as $(0,2,0)$. It nevertheless cannot occur in a full decomposition with the same complete signature because the labels prescribed to layers $0$ and $2$ already fill those layers. This is the capacity mechanism formalized by \eqref{eq:induction-signature}.
\end{remark}

\subsection{Exact partition slack}

For a tree assignment $F=(F_e)_{e\in E(H)}$, define the assignment slack by $\Lambda_F(\PP)\defeq\sum_{e\in E(H)}(\cross_{F_e}(\PP)-p_{\PP}(e)+1)$. It is nonnegative by \Cref{lem:tree-cross-lower}. If the assigned graph $G=\biguplus_{e\in E(H)}F_e$ has a full ordered $k$-tree decomposition, define its graphical partition surplus by $B_G(\PP)\defeq\cross_G(\PP)-k(|\PP|-1)$. Since every spanning tree crosses a partition with $|\PP|$ blocks in at least $|\PP|-1$ edges, one also has $B_G(\PP)=\sum_{i=0}^{k-1}\cross_{T_i}(\PP)-k(|\PP|-1)\ge0$. The first expression shows that $B_G(\PP)$ depends only on $G$, not on the chosen full decomposition.

\begin{proposition}[Exact slack identity]\label{prop:slack-identity}
For every partition $\PP$ of $V$,
\begin{equation}\label{eq:slack-identity}
 w_H(\PP)-k(|\PP|-1)=B_G(\PP)-\Lambda_F(\PP).
\end{equation}
Consequently, for a fixed assignment whose graph admits a full $k$-tree decomposition, $H$ is $k$-WPC if and only if $\Lambda_F(\PP)\le B_G(\PP)$ for every partition $\PP$.
\end{proposition}

\begin{proof}
By the definition of $\Lambda_F$, $w_H(\PP)=\sum_e(p_{\PP}(e)-1)=\sum_e\cross_{F_e}(\PP)-\Lambda_F(\PP)=\cross_G(\PP)-\Lambda_F(\PP)$. Subtracting $k(|\PP|-1)$ and using the definition of $B_G$ gives \eqref{eq:slack-identity}. The stated equivalence is exactly the definition of $k$-WPC\@.
\end{proof}

For a star assignment, the local slack has a closed form.

\begin{lemma}[Star crossing formula]\label{lem:star-crossing}
Let $S_e(c)$ be the star on $e$ centered at $c$, and let $P_c$ be the block of $\PP$ containing $c$. Then
\begin{equation}\label{eq:star-crossing}
 \cross_{S_e(c)}(\PP)
 =p_{\PP}(e)-1+
 \sum_{P\in\PP\setminus\{P_c\}}\max\{0,|P\cap e|-1\}.
\end{equation}
\end{lemma}

\begin{proof}
An edge $cu$ crosses $\PP$ precisely when $u\notin P_c$, so $\cross_{S_e(c)}(\PP)=\sum_{P\ne P_c}|P\cap e|$. Exactly $p_{\PP}(e)-1$ blocks distinct from $P_c$ meet $e$. Splitting each nonzero summand as $|P\cap e|=1+(|P\cap e|-1)$ gives \eqref{eq:star-crossing}.
\end{proof}

\begin{definition}[Partition-tight realization]\label{def:partition-tight}
A tree assignment is \emph{partition-tight} if $\sum_e\cross_{F_e}(\PP)=w_H(\PP)$ for every partition $\PP$. Equivalently, every individual inequality in \Cref{lem:tree-cross-lower} is an equality for every hyperedge occurrence and every partition.
\end{definition}

The equivalence follows because all local slacks are nonnegative. Partition-tightness is not needed to pass from hypergraph weak partition connectivity to graph tree packing: the inequality in \eqref{eq:graph-dominates-hypergraph} has the correct direction. It would be relevant in the reverse direction, if one attempted to infer the hypergraph inequalities solely from graph crossings of a fixed assignment. Without control of $\Lambda_F$, graph crossings may strictly exceed hypergraph weak excess.

If $|e|\ge3$, the assigned tree $F_e$ cannot attain equality in \Cref{lem:tree-cross-lower} for every partition. Indeed, $F_e$ has a vertex $c$ of degree at least two. For the bipartition $\{\{c\},V\setminus\{c\}\}$, one has $p_{\PP}(e)-1=1$, whereas $F_e$ has at least two crossing edges. The layer-contained proof avoids this reverse-transfer obstruction by applying quotient rank separately in each tree layer.

If every assigned tree is a star centered at $c(e)$, then \Cref{lem:star-crossing} gives
\begin{equation}\label{eq:star-global-lambda}
 \Lambda_F(\PP)=
 \sum_{e\in E(H)}\sum_{P\in\PP\setminus\{P_{c(e)}\}}
 \max\{0,|P\cap e|-1\}.
\end{equation}
Together with \Cref{prop:slack-identity}, this shows that weak partition connectivity of a fixed star assignment whose graph has a full decomposition is equivalent to the finite family of inequalities
\begin{equation}\label{eq:finite-system}
 \sum_{e\in E(H)}\sum_{P\in\PP\setminus\{P_{c(e)}\}}\max\{0,|P\cap e|-1\}
 \le\cross_G(\PP)-k(|\PP|-1)
 \qquad(\PP\in\Pi(V)),
\end{equation}
where $\Pi(V)$ is the set of all set partitions of $V$. For a layer-contained star assignment, these inequalities follow automatically from \Cref{thm:layer-positive}; the quotient-rank argument proves them without enumerating partitions.

For the discrete partition $\PP_0$, every assigned tree edge crosses, and therefore $\Lambda_F(\PP_0)=\sum_e(|E(F_e)|-(|e|-1))=0$. Hence $w_H(\PP_0)-k(|\PP_0|-1)=\rho(H)-k(t-1)$. The discrete partition has no assignment slack with which to absorb overfullness; this is the rigidity behind \Cref{thm:excess-obstruction}.

\section{Corrected formulations, pruning, and conclusions}\label{sec:formulations}

The edge-count obstruction leads to three distinct universal statements. Their quantifiers and logical relation must be kept separate.

\begin{enumerate}[label=\textup{(\Alph*)},leftmargin=2.4em]
\item \textbf{Literal full-assignment statement.} Every $k$-WPC hypergraph has a $k$-distinguishable full tree assignment. This is \Cref{conj:literal} and is false by \Cref{thm:full-edge-counterexamples}.

\item \textbf{Critical full-assignment statement.} Every $k$-critical, $k$-WPC hypergraph has a $k$-distinguishable full tree assignment. This is \Cref{conj:critical}. It is compatible with edge counting and is sufficient for the intersection-matrix implication in \Cref{thm:critical-implies-matrix}. Under the canonical support representation, its input conditions coincide with the intersection weight conditions by \Cref{prop:critical-weight-equivalence}; its distinguishability conclusion is a stronger structural certificate than matrix full column rank alone.

\item \textbf{Pruned-assignment statement.} Every $k$-WPC hypergraph $H$ has a tree assignment $G$ and a labelled spanning subgraph $G'\subseteq G$ with $|E(G')|=k(t-1)$ such that $G'$ has a signature-unique ordered $k$-tree decomposition, where, for each label $e$, the signature coordinate is computed using the retained class $E(F_e)\cap E(G')$.
\end{enumerate}

Let \textup{(C$_>$)} denote statement \textup{(C)} restricted to overfull hypergraphs, namely those satisfying $\rho(H)>k(t-1)$.

\begin{proposition}[Exact logical decomposition of pruning]\label{prop:pruning-logic}
As universal statements over $k$-WPC hypergraphs, \textup{(C)} holds if and only if both \textup{(B)} and \textup{(C$_>$)} hold. In particular, the unrestricted pruning statement implies the critical full-assignment conjecture; it is not logically independent of it.
\end{proposition}

\begin{proof}
Assume \textup{(C)} and let $H$ be critical. Every full tree assignment $G$ of $H$ has $|E(G)|=\rho(H)=k(t-1)$ by \Cref{lem:counts}. The subgraph $G'$ supplied by \textup{(C)} has the same number of edge occurrences as $G$ and is obtained by deletion, so $G'=G$. Hence \textup{(B)} holds. The restriction \textup{(C$_>$)} follows immediately from \textup{(C)}.

Conversely, assume \textup{(B)} and \textup{(C$_>$)}, and let $H$ be $k$-WPC\@. By \Cref{lem:discrete}, $\rho(H)\ge k(t-1)$. If equality holds, apply \textup{(B)} and take $G'=G$. If the inequality is strict, apply \textup{(C$_>$)}. These are the only two cases, so \textup{(C)} holds.
\end{proof}

The overfull component cannot be solved by a universal rule that merely deletes whole hyperedge occurrences until the target excess is reached; whole-hyperedge contributions can be arithmetically too coarse.

\begin{example}[Whole-hyperedge pruning can miss the target]\label{ex:whole-edge-pruning}
Let $V=\{1,2,3,4\}$, let $k=1$, and let $H$ have two hyperedges $e_1=\{1,2,3\}$ and $e_2=\{1,2,4\}$. Then $\rho(H)=4>3=k(|V|-1)$, but no subhypergraph has excess $3$: the possible excesses are $0$, $2$, and $4$.

The hypergraph is nevertheless $1$-WPC\@. Let $\PP$ have $s$ blocks. Since $e_1\cup e_2=V$, every block is met by at least one of the two hyperedges. The block containing vertex $1$ is met by both. Counting block--hyperedge incidences gives $p_{\PP}(e_1)+p_{\PP}(e_2)\ge s+1$, and hence $w_H(\PP)=p_{\PP}(e_1)+p_{\PP}(e_2)-2\ge s-1$.

Thus deletion of whole hyperedge occurrences alone cannot be a valid universal exact-pruning rule. Any successful exact theorem must use a finer operation, such as deletion of assigned graph edges or determinant rows, or some other mechanism that is not constrained to change the excess by an entire hyperedge contribution.
\end{example}

The logarithmic relaxation recorded in~\cite[Theorem~6.4]{GLSTW24}, based on results of Cheriyan--Salavatipour and Calinescu--Chekuri--Vondr\'ak~\cite{CS07,CCV09}, yields a distinguishable subgraph under stronger weak partition connectivity. It motivates \textup{(C$_>$)} but does not settle the exact $k$-WPC formulation.

The preceding results give the following corrected picture. The literal conjecture fails because weak partition connectivity supplies a lower bound on excess, while a full decomposition imposes equality. The family $H_{t,k,q}$ gives counterexamples for every $t\ge2$, $k\ge1$, and $q\ge1$. In the critical regime, every tree assignment already has a full $k$-tree decomposition by \Cref{thm:any-assignment-packing}; the remaining assertion is whether some assignment has a decomposition with singleton signature fiber. Layer-contained certificates establish this assertion for an explicit infinite class, with no star requirement and no restriction on the ranks of interior labels. Their weak partition inequalities follow from quotient rank, their signatures are isolated by a layer-by-layer capacity induction, and the construction is stable under iterated one-vertex sums.

For the Reed--Solomon intersection-matrix route, the critical equality is $\wt(I_{[t]})=k(t-1)$. By \Cref{prop:critical-weight-equivalence}, the corresponding proper-subset weight inequalities are exactly the canonical form of critical weak partition connectivity. By \Cref{prop:square}, the total equality is precisely the condition under which a full tree assignment supplies the $k(t-1)$ selected bottom rows required for the square minor. The decomposition-indexed generalized-Laplace expansion in \Cref{lem:det-expansion} then identifies the signature of an ordered decomposition with the exponent vector of its determinant monomial, while its coefficient formula records the possible cancellation among decompositions in the same signature fiber.

The remaining overfull problem is the separate component \textup{(C$_>$)}: if $H$ is $k$-WPC and $\rho(H)>k(t-1)$, must some tree assignment of $H$ contain a labelled spanning subgraph with exactly $k(t-1)$ edges and a signature-unique ordered $k$-tree decomposition? By \Cref{prop:pruning-logic}, a universal pruned-assignment theorem consists exactly of the critical conjecture together with this overfull component. No claim of logical independence between them is justified.

\section*{Declaration of Generative AI and AI-Assisted Technologies in the Writing Process}
During the preparation of this manuscript, the authors used a DeepSeek-based specialized agent to generate an initial proof draft. The authors subsequently verified and revised all arguments and accept full responsibility for the contents of the manuscript.

\end{document}